\documentclass{amsart}
\usepackage{amssymb}
\usepackage{amsfonts}

\newtheorem{theorem}{Theorem}
\theoremstyle{plain}

\newtheorem{corollary}{Corollary}

\newtheorem{lemma}{Lemma}

\newtheorem{proposition}{Proposition}
\newtheorem{remark}{Remark}

\numberwithin{equation}{section}
\input{tcilatex}

\begin{document}
\title[Young type inequalities]{New Estimates on Integral Inequalities and
Their Applications}
\author{M.Emin \"{O}zdemir$^{\blacklozenge }$}
\address{$^{\blacklozenge }$Atat\"{u}rk University, K.K. Education Faculty,
Department of Mathematics, Erzurum 25240, Turkey}
\email{emos@atauni.edu.tr}
\author{Mustafa G\"{u}rb\"{u}z$^{\spadesuit }$}
\address{$^{\spadesuit }$A\u{g}r\i\ \.{I}brahim \c{C}e\c{c}en University,
Faculty of Education, Department of Mathematics, A\u{g}r\i\ 04100, Turkey}
\email{mgurbuz@agri.edu.tr}
\thanks{$^{\spadesuit }$Corresponding Author}
\author{Mevl\"{u}t Tun\c{c}$^{\clubsuit }$}
\address{$^{\clubsuit }$Kilis 7 Aral\i k University, Faculty of Science and
Arts, Department of Mathematics, Kilis\ 79100, Turkey}
\email{mevluttunc@kilis.edu.tr}
\date{November 30, 2012}
\subjclass[2000]{ Primary 26D15, Secondary 26D10}
\keywords{Young inequality, H\"{o}lder inequality, Power-mean inequality.}

\begin{abstract}
In this paper, we obtain some inequalities by using a kernel and an
inequality which is a result of Young inequality. Besides we give some
applications to special means.
\end{abstract}

\maketitle

\section{\protect\bigskip Introduction}

Let $f:I\subseteq 
\mathbb{R}
\rightarrow 
\mathbb{R}
$ be a convex function on the interval $I$ of real numbers and $a,b\in I$
with $a<b$. The inequality%
\begin{equation}
f\left( \frac{a+b}{2}\right) \leq \frac{1}{b-a}\int_{a}^{b}f\left( x\right)
dx\leq \frac{f\left( a\right) +f\left( b\right) }{2}  \label{hh}
\end{equation}%
is well known in the literature as Hermite-Hadamard's inequality for convex
functions \cite{SC}.

In \cite{SR}, Dragomir and Agarwal proved one lemma and some Hadamard's type
inequalities for convex functions as following:

\begin{lemma}
\label{l}\bigskip Let $f:I^{\circ }\subseteq 
\mathbb{R}
\rightarrow 
\mathbb{R}
$ be a differantiable mapping on $I^{\circ },$ $a,b\in I^{\circ }$ with $%
a<b. $ If $f^{\prime }\in L\left[ a,b\right] ,$ then the following equality
holds:%
\begin{equation*}
\frac{f\left( a\right) +f\left( b\right) }{2}-\frac{1}{b-a}%
\int_{a}^{b}f\left( x\right) dx=\frac{b-a}{2}\int_{0}^{1}\left( 1-2t\right)
f^{\prime }\left( ta+\left( 1-t\right) b\right) dt.
\end{equation*}
\end{lemma}

\begin{theorem}
Let $f:I^{\circ }\subseteq 
\mathbb{R}
\rightarrow 
\mathbb{R}
$ be a differantiable mapping on $I^{\circ },$ $a,b\in I^{\circ }$ with $%
a<b. $ If $\left\vert f^{\prime }\right\vert $ is convex on $\left[ a,b%
\right] ,$ then the following inequality holds:%
\begin{equation*}
\left\vert \frac{f\left( a\right) +f\left( b\right) }{2}-\frac{1}{b-a}%
\int_{a}^{b}f\left( x\right) dx\right\vert \leq \frac{\left( b-a\right)
\left( \left\vert f^{\prime }\left( a\right) \right\vert +\left\vert
f^{\prime }\left( b\right) \right\vert \right) }{8}.
\end{equation*}
\end{theorem}

\begin{theorem}
Let $f:I^{\circ }\subseteq 
\mathbb{R}
\rightarrow 
\mathbb{R}
$ be a differantiable mapping on $I^{\circ },$ $a,b\in I^{\circ }$ with $%
a<b, $ and let $p>1.$ If the new mapping $\left\vert f^{\prime }\right\vert
^{p/\left( p-1\right) }$ is convex on $\left[ a,b\right] ,$ then the
following inequality holds:%
\begin{equation*}
\left\vert \frac{f\left( a\right) +f\left( b\right) }{2}-\frac{1}{b-a}%
\int_{a}^{b}f\left( x\right) dx\right\vert \leq \frac{\left( b-a\right) }{%
2\left( p+1\right) ^{1/p}}\left[ \frac{\left( \left\vert f^{\prime }\left(
a\right) \right\vert ^{p/\left( p-1\right) }+\left\vert f^{\prime }\left(
b\right) \right\vert ^{p/\left( p-1\right) }\right) }{2}\right] ^{\left(
p-1\right) /p}.
\end{equation*}
\end{theorem}

\bigskip We recall the well-known Young's inequality which can be stated as
follows.

\begin{theorem}
(\textbf{Young's inequality}, see \cite{2}, p. 117) If $a,b>0$ and $p,q>1$
satisfy $\frac{1}{p}+\frac{1}{q}=1,$ then%
\begin{equation}
ab\leq \frac{a^{p}}{p}+\frac{b^{q}}{q}.  \label{yo}
\end{equation}%
Equality holds if and only if $a^{p}=b^{q}.$
\end{theorem}

\begin{remark}
\cite{MEV} \bigskip If we take $a=t^{\frac{1-p}{p^{2}}}$ and $b=t^{\frac{1}{%
pq}}$ in (\ref{yo}), we have%
\begin{equation}
1\leq \frac{1}{p}t^{\frac{1}{p}-1}+\left( 1-\frac{1}{p}\right) t^{\frac{1}{p}%
}  \label{hay}
\end{equation}%
for all $t\in \left( 0,1\right) .$
\end{remark}

Chebyshev inequality is given in the following theorem.

\begin{theorem}
Let $f,g:\left[ a,b\right] \rightarrow 
\mathbb{R}
$\ be integrable functions, both increasing or both decreasing. Furthermore,
let $p:\left[ a,b\right] \rightarrow 
\mathbb{R}
_{0}^{+}$\ be an integrable function. Then
\end{theorem}

\begin{equation}
\int_{a}^{b}p\left( x\right) dx\int_{a}^{b}p\left( x\right) f\left( x\right)
g\left( x\right) dx\geq \int_{a}^{b}p\left( x\right) f\left( x\right)
dx\int_{a}^{b}p\left( x\right) g\left( x\right) dx.  \label{1.7}
\end{equation}

If one of the functions $f$ and $g$ is nonincreasing and the other is
nondecreasing, then the inequality in (\ref{1.7}) is reversed. Inequality (%
\ref{1.7}) is known in the literature as Chebyshev inequality and so are the
following special cases of (\ref{1.7}):%
\begin{equation}
\int_{a}^{b}f\left( x\right) g\left( x\right) dx\geq \frac{1}{b-a}%
\int_{a}^{b}f\left( x\right) dx\int_{a}^{b}g\left( x\right) dx  \label{1.8}
\end{equation}%
and 
\begin{equation}
\int_{0}^{1}f\left( x\right) g\left( x\right) dx\geq \int_{0}^{1}f\left(
x\right) dx\int_{0}^{1}g\left( x\right) dx.  \label{1.9}
\end{equation}

\bigskip

In the following sections our main results are given.

\bigskip

\section{\protect\bigskip New Results}

\begin{theorem}
\label{t1}Let $f:I^{\circ }\subseteq 
\mathbb{R}
\rightarrow 
\mathbb{R}
$ be a differantiable convex mapping on $I^{\circ },$ $a,b\in I^{\circ }$
with $a<b.$ If $f,f^{\prime }\in L\left[ a,b\right] ,$ for $p,q>1,$ $\frac{1%
}{p}+\frac{1}{q}=1,$ the following inequality holds:%
\begin{equation*}
\left\vert \frac{f\left( a\right) +f\left( b\right) }{2}-\frac{1}{b-a}%
\int_{a}^{b}f\left( x\right) dx\right\vert \leq \frac{\left( b-a\right) ^{%
\frac{1}{p}}p}{\left( p+1\right) ^{1+\frac{1}{p}}}\left(
\int_{a}^{b}\left\vert f^{\prime }\left( x\right) \right\vert ^{q}dt\right)
^{\frac{1}{q}}
\end{equation*}
\end{theorem}

\begin{proof}
By Lemma \ref{l} in \cite{SR} we have%
\begin{equation}
\frac{f\left( a\right) +f\left( b\right) }{2}-\frac{1}{b-a}%
\int_{a}^{b}f\left( x\right) dx=\frac{b-a}{2}\int_{0}^{1}\left( 1-2t\right)
f^{\prime }\left( ta+\left( 1-t\right) b\right) dt.  \label{1}
\end{equation}%
As we choose $f$ is convex on $I^{\circ }$ by using Hadamard's inequality,
we can see that both sides are positive of Lemma 2.1

On the other hand by using Young's inequality we have $\left( t\in \left[ 0,1%
\right] ,\text{ }p>1\right) $%
\begin{equation*}
1\leq \frac{1}{p}t^{\frac{1}{p}-1}+\left( 1-\frac{1}{p}\right) t^{\frac{1}{p}%
}
\end{equation*}%
which is proved in \cite{MEV}. If we integrate both sides of above
inequality respect to $t$ over $\left[ 0,1\right] $\ we get%
\begin{equation}
1\leq \int_{0}^{1}\left( \frac{1}{p}t^{\frac{1}{p}-1}+\left( 1-\frac{1}{p}%
\right) t^{\frac{1}{p}}\right) dt  \label{2}
\end{equation}

By multipliying both sides of (\ref{1}) and (\ref{2}) we have%
\begin{eqnarray}
\frac{f\left( a\right) +f\left( b\right) }{2}-\frac{1}{b-a}%
\int_{a}^{b}f\left( x\right) dx &\leq &\frac{b-a}{2}\int_{0}^{1}\left( \frac{%
1}{p}t^{\frac{1}{p}-1}+\left( 1-\frac{1}{p}\right) t^{\frac{1}{p}}\right) dt
\notag \\
&&\times \int_{0}^{1}\left( 1-2t\right) f^{\prime }\left( ta+\left(
1-t\right) b\right) dt  \label{3}
\end{eqnarray}%
To use H\"{o}lder's inequality we apply properties of ablosute value as%
\begin{eqnarray*}
\left\vert \frac{f\left( a\right) +f\left( b\right) }{2}-\frac{1}{b-a}%
\int_{a}^{b}f\left( x\right) dx\right\vert &\leq &\frac{b-a}{2}%
\int_{0}^{1}\left( \frac{1}{p}t^{\frac{1}{p}-1}+\left( 1-\frac{1}{p}\right)
t^{\frac{1}{p}}\right) dt \\
&&\times \int_{0}^{1}\left\vert \left( 1-2t\right) f^{\prime }\left(
ta+\left( 1-t\right) b\right) \right\vert dt
\end{eqnarray*}%
By using H\"{o}lder's inequality we have%
\begin{eqnarray*}
&&\left\vert \frac{f\left( a\right) +f\left( b\right) }{2}-\frac{1}{b-a}%
\int_{a}^{b}f\left( x\right) dx\right\vert \\
&\leq &\frac{b-a}{2}\int_{0}^{1}\left( \frac{1}{p}t^{\frac{1}{p}-1}+\left( 1-%
\frac{1}{p}\right) t^{\frac{1}{p}}\right) dt \\
&&\times \left( \int_{0}^{1}\left\vert 1-2t\right\vert ^{p}dt\right) ^{\frac{%
1}{p}}\left( \int_{0}^{1}\left\vert f^{\prime }\left( ta+\left( 1-t\right)
b\right) \right\vert ^{q}dt\right) ^{\frac{1}{q}} \\
&=&\frac{b-a}{2}\left( \frac{2p}{p+1}\right) \left( \int_{0}^{\frac{1}{2}%
}\left( 1-2t\right) ^{p}dt+\int_{\frac{1}{2}}^{1}\left( 2t-1\right)
^{p}dt\right) ^{\frac{1}{p}} \\
&&\times \left( \int_{0}^{1}\left\vert f^{\prime }\left( ta+\left(
1-t\right) b\right) \right\vert ^{q}dt\right) ^{\frac{1}{q}} \\
&=&\frac{b-a}{2}\left( \frac{2p}{p+1}\right) \left( \frac{1}{p+1}\right) ^{%
\frac{1}{p}}\left( \int_{0}^{1}\left\vert f^{\prime }\left( ta+\left(
1-t\right) b\right) \right\vert ^{q}dt\right) ^{\frac{1}{q}}.
\end{eqnarray*}%
By simple calculation we get the desired result.
\end{proof}

\begin{corollary}
If we choose $p=q=2$ in \textit{Theorem \ref{t1}, we have}%
\begin{equation*}
\left\vert \frac{f\left( a\right) +f\left( b\right) }{2}-\frac{1}{b-a}%
\int_{a}^{b}f\left( x\right) dx\right\vert \leq \frac{2\left( b-a\right) ^{%
\frac{1}{2}}}{3^{\frac{3}{2}}}\left( \int_{a}^{b}\left\vert f^{\prime
}\left( x\right) \right\vert ^{2}dt\right) ^{\frac{1}{2}}.
\end{equation*}
\end{corollary}

\begin{theorem}
\label{t2}Let $f:I^{\circ }\subseteq 
\mathbb{R}
\rightarrow 
\mathbb{R}
$ be a differantiable convex mapping on $I^{\circ },$ $a,b\in I^{\circ }$
with $a<b.$ If $f,f^{\prime }\in L\left[ a,b\right] ,$ for $p,q>1,$ $\frac{1%
}{p}+\frac{1}{q}=1,$ the following inequality holds:%
\begin{equation*}
\frac{f\left( a\right) +f\left( b\right) }{2}-\frac{1}{b-a}%
\int_{a}^{b}f\left( x\right) dx\leq \frac{p\left( p-1\right) }{\left(
p+1\right) \left( 2p+1\right) }\left[ f\left( b\right) -f\left( a\right) %
\right]
\end{equation*}
\end{theorem}

\begin{proof}
The same steps are followed as in Theorem \ref{t1} until (\ref{3}). Then we
know that the function%
\begin{equation*}
f\left( t\right) =1-2t
\end{equation*}%
is nonincreasing on $\left[ 0,1\right] ,$ and since $f$ is convex, $%
f^{\prime \prime }$ is positive on $I^{\circ }$. So $f^{\prime }$ is
nondecreasing. By using these phrases we can use Chebyshev inequality as:%
\begin{eqnarray*}
&&\frac{f\left( a\right) +f\left( b\right) }{2}-\frac{1}{b-a}%
\int_{a}^{b}f\left( x\right) dx \\
&\leq &\frac{b-a}{2}\int_{0}^{1}\left( \frac{1}{p}t^{\frac{1}{p}-1}+\left( 1-%
\frac{1}{p}\right) t^{\frac{1}{p}}\right) dt \\
&&\times \int_{0}^{1}\left( 1-2t\right) f^{\prime }\left( ta+\left(
1-t\right) b\right) dt \\
&\leq &\frac{b-a}{2}\int_{0}^{1}\left( \frac{1}{p}t^{\frac{1}{p}-1}+\left( 1-%
\frac{1}{p}\right) t^{\frac{1}{p}}\right) dt \\
&&\times \int_{0}^{1}\left( 1-2t\right) dt\int_{0}^{1}f^{\prime }\left(
ta+\left( 1-t\right) b\right) dt \\
&=&\frac{f\left( b\right) -f\left( a\right) }{2}\int_{0}^{1}\left( \frac{1}{p%
}t^{\frac{1}{p}-1}+\left( 1-\frac{1}{p}\right) t^{\frac{1}{p}}\right) dt \\
&&\times \int_{0}^{1}\left( 1-2t\right) dt
\end{eqnarray*}%
Since both of the functions $\left( \frac{1}{p}t^{\frac{1}{p}-1}+\left( 1-%
\frac{1}{p}\right) t^{\frac{1}{p}}\right) $ and $\left( 1-2t\right) $ are
nonincreasing, we can use Chebyshev inequality again as:%
\begin{equation*}
\frac{f\left( a\right) +f\left( b\right) }{2}-\frac{1}{b-a}%
\int_{a}^{b}f\left( x\right) dx\leq \frac{f\left( b\right) -f\left( a\right) 
}{2}\int_{0}^{1}\left( \frac{1}{p}t^{\frac{1}{p}-1}+\left( 1-\frac{1}{p}%
\right) t^{\frac{1}{p}}\right) \left( 1-2t\right) dt
\end{equation*}%
By simple calculation, the proof is completed.
\end{proof}

\begin{corollary}
If we choose $p=1,1$ in \textit{Theorem \ref{t2}, we have}%
\begin{equation*}
\frac{f\left( a\right) +f\left( b\right) }{2}-\frac{1}{b-a}%
\int_{a}^{b}f\left( x\right) dx\leq \frac{11}{483}\left[ f\left( b\right)
-f\left( a\right) \right] .
\end{equation*}
\end{corollary}

\begin{theorem}
\label{t3}Let $f:I^{\circ }\subseteq 
\mathbb{R}
\rightarrow 
\mathbb{R}
$ be a differantiable convex mapping on $I^{\circ },$ $a,b\in I^{\circ }$
with $a<b.$ If $f,f^{\prime }\in L\left[ a,b\right] ,$ for $p,q>1,$ $\frac{1%
}{p}+\frac{1}{q}=1,$ the following inequality holds:%
\begin{equation*}
\left\vert \frac{f\left( a\right) +f\left( b\right) }{2}-\frac{1}{b-a}%
\int_{a}^{b}f\left( x\right) dx\right\vert \leq \frac{2^{\frac{1}{q}}p}{%
\left( p+1\right) \left( b-a\right) }\left( \int_{a}^{b}\left\vert x-\frac{%
a+b}{2}\right\vert \left\vert f^{\prime }\left( x\right) \right\vert
^{q}dx\right) ^{\frac{1}{q}}.
\end{equation*}
\end{theorem}

\begin{proof}
The same steps are followed as in Theorem \ref{t1} until (\ref{3}). Then by
using convexity of $f$ and properties of absolute value we have%
\begin{eqnarray*}
\left\vert \frac{f\left( a\right) +f\left( b\right) }{2}-\frac{1}{b-a}%
\int_{a}^{b}f\left( x\right) dx\right\vert &\leq &\frac{b-a}{2}%
\int_{0}^{1}\left( \frac{1}{p}t^{\frac{1}{p}-1}+\left( 1-\frac{1}{p}\right)
t^{\frac{1}{p}}\right) dt \\
&&\times \int_{0}^{1}\left\vert \left( 1-2t\right) f^{\prime }\left(
ta+\left( 1-t\right) b\right) \right\vert dt.
\end{eqnarray*}%
By using Power-mean inequality we have%
\begin{eqnarray}
&&\left\vert \frac{f\left( a\right) +f\left( b\right) }{2}-\frac{1}{b-a}%
\int_{a}^{b}f\left( x\right) dx\right\vert  \label{m} \\
&\leq &\frac{b-a}{2}\int_{0}^{1}\left( \frac{1}{p}t^{\frac{1}{p}-1}+\left( 1-%
\frac{1}{p}\right) t^{\frac{1}{p}}\right) dt  \notag \\
&&\times \left( \int_{0}^{1}\left\vert 1-2t\right\vert dt\right) ^{\frac{1}{p%
}}\left( \int_{0}^{1}\left\vert \left( 1-2t\right) \right\vert \left\vert
f^{\prime }\left( ta+\left( 1-t\right) b\right) \right\vert ^{q}dt\right) ^{%
\frac{1}{q}}  \notag \\
&=&\frac{2^{\frac{1}{q}}p}{p+1}\left( \int_{0}^{1}\left\vert \left(
1-2t\right) \right\vert \left\vert f^{\prime }\left( ta+\left( 1-t\right)
b\right) \right\vert ^{q}dt\right) ^{\frac{1}{q}}  \notag
\end{eqnarray}

And using the change of the variable $x=ta+\left( 1-t\right) b,$ $t\in \left[
0,1\right] $, inequality (\ref{m}) can be writen as%
\begin{equation*}
\left\vert \frac{f\left( a\right) +f\left( b\right) }{2}-\frac{1}{b-a}%
\int_{a}^{b}f\left( x\right) dx\right\vert \leq \frac{2^{\frac{1}{q}}p}{%
\left( p+1\right) \left( b-a\right) }\left( \int_{a}^{b}\left\vert x-\frac{%
a+b}{2}\right\vert \left\vert f^{\prime }\left( x\right) \right\vert
^{q}dt\right) ^{\frac{1}{q}}.
\end{equation*}

Then the proof is completed.
\end{proof}

\begin{corollary}
If we choose $p=q=2$ in \textit{Theorem \ref{t3}, we have}%
\begin{equation*}
\left\vert \frac{f\left( a\right) +f\left( b\right) }{2}-\frac{1}{b-a}%
\int_{a}^{b}f\left( x\right) dx\right\vert \leq \frac{2^{\frac{3}{2}}}{3}%
\left( \int_{0}^{1}\left\vert x-\frac{a+b}{2}\right\vert \left\vert
f^{\prime }\left( x\right) \right\vert ^{q}dx\right) ^{\frac{1}{2}}
\end{equation*}
\end{corollary}

\section{\protect\bigskip Applications to special means}

We now consider the applications of our Theorems to the following special
means

The arithmetic mean: $A=A\left( a,b\right) :=\frac{a+b}{2},$\ \ $a,b\geq 0,$

The geometric mean: $G=G\left( a,b\right) :=\sqrt{ab},$ \ $a,b\geq 0,$

The harmonic mean: $H=H\left( a,b\right) :=\frac{2ab}{a+b},$ \ $a,b\geq 0,$

The logarithmic mean: $L=L\left( a,b\right) :=\left\{ 
\begin{array}{l}
a\text{ \ \ \ \ \ \ \ \ \ \ \ \ if \ \ }a=b \\ 
\frac{b-a}{\ln b-\ln a}\text{ \ \ \ \ \ if \ \ }a\neq b%
\end{array}%
\right. ,$ \ $a,b\geq 0,$

The Identric mean: $I=I\left( a,b\right) :=\left\{ 
\begin{array}{l}
a\text{ \ \ \ \ \ \ \ \ \ \ \ \ \ \ \ \ \ if \ \ }a=b \\ 
\frac{1}{e}\left( \frac{b^{b}}{a^{a}}\right) ^{\frac{1}{b-a}}\text{ \ \ \ \
\ if \ \ }a\neq b%
\end{array}%
\right. ,$ \ $a,b\geq 0,$

The p-logarithmic mean:$L_{p}=L_{p}\left( a,b\right) :=\left\{ 
\begin{array}{l}
\left[ \frac{b^{p+1}-a^{p+1}}{\left( p+1\right) \left( b-a\right) }\right]
^{1/p}\text{ \ \ \ \ \ if \ \ }a\neq b \\ 
a\text{ \ \ \ \ \ \ \ \ \ \ \ \ \ \ \ \ \ \ \ \ \ \ if \ \ }a=b%
\end{array}%
\right. ,$ \ $\ \ \ p\in 
\mathbb{R}
\backslash \left\{ -1,0\right\} ;$ \ $a,b>0.$

\bigskip The following inequality is well known in the literature:%
\begin{equation*}
H\leq G\leq L\leq I\leq A\leq K
\end{equation*}

It is also known that $L_{p}$ is monotonically increasing over $p\in 
\mathbb{R}
,$ denoting $L_{1}=A,$ $L_{0}=I$ and $L_{-1}=L.$

The following propositions holds:

\bigskip

\begin{proposition}
Let $a$,$b\in 
\mathbb{R}
^{+},$ $a<b$ and $n\in 
\mathbb{N}
,$ $n\geq 2.$ Then, we have%
\begin{equation}
\left\vert A\left( a^{n},b^{n}\right) -L_{n}\left( a,b\right) \right\vert
\leq \frac{n.p.\left( b-a\right) }{\left( p+1\right) ^{1+\frac{1}{p}}}%
A^{n-1}\left( a,b\right) .  \label{31}
\end{equation}
\end{proposition}

\begin{proof}
The proof is immediate from Theorem \ref{t1} applied for $f(x)=x^{n}$, $x\in 
\mathbb{R}
$.
\end{proof}

\bigskip

\begin{proposition}
Let $a$,$b\in 
\mathbb{R}
^{+},$ $a<b$ and $n\in 
\mathbb{N}
,$ $n\geq 2.$ Then, for all $p>1$, the following inequality holds:%
\begin{equation}
A\left( a^{n},b^{n}\right) -L_{n}\left( a,b\right) \leq \frac{p\left(
p-1\right) }{\left( p+1\right) \left( 2p+1\right) }\left[ b^{n}-a^{n}\right]
.  \label{32}
\end{equation}
\end{proposition}

\begin{proof}
The proof is immediate from Theorem \ref{t2} applied for $f(x)=x^{n}$, $x\in 
\mathbb{R}
$.
\end{proof}

\bigskip

\begin{proposition}
Let $a$,$b\in 
\mathbb{R}
^{+},$ $a<b$ and $n\in 
\mathbb{N}
,$ $n\geq 2.$ Then, for all $p>1$, the following inequality holds:%
\begin{equation}
\left\vert A\left( a^{n},b^{n}\right) -L_{n}\left( a,b\right) \right\vert
\leq \frac{2^{\frac{1}{q}}p}{\left( p+1\right) \left( b-a\right) }\left(
\int_{a}^{b}\left\vert x-\frac{a+b}{2}\right\vert \left\vert f^{\prime
}\left( x\right) \right\vert ^{q}dx\right) ^{\frac{1}{q}}.  \label{33}
\end{equation}
\end{proposition}

\begin{proof}
The proof is immediate from Theorem \ref{t3} applied for $f(x)=x^{n}$, $x\in 
\mathbb{R}
$.
\end{proof}

\bigskip

\begin{proposition}
Let $a$,$b\in 
\mathbb{R}
^{+},$ $a<b.$ Then, we have%
\begin{equation*}
\left\vert A\left( a^{-1},b^{-1}\right) -L^{-1}\left( a,b\right) \right\vert
\leq \frac{p\left( b-a\right) ^{\frac{1}{p}}}{\left( p+1\right) ^{1+\frac{1}{%
p}}}\left( \int_{a}^{b}\left\vert x\right\vert ^{-2q}dx\right) ^{\frac{1}{q}%
}.
\end{equation*}
\end{proposition}

\begin{proof}
The proof is obvious from Theorem \ref{t1} applied for $f(x)=1/x$, $x\in
\lbrack a,b]$.
\end{proof}

$\bigskip $

\begin{proposition}
Let $a$,$b\in 
\mathbb{R}
^{+},$ $a<b.$ Then, we have%
\begin{equation*}
A\left( a^{-1},b^{-1}\right) -L^{-1}\left( a,b\right) \leq \frac{2p\left(
p-1\right) }{\left( p+1\right) \left( 2p+1\right) }H^{-1}\left( a,b\right) .
\end{equation*}
\end{proposition}

\begin{proof}
The proof is obvious from Theorem \ref{t2} applied for $f(x)=1/x$, $x\in
\lbrack a,b]$.
\end{proof}

\bigskip

\begin{proposition}
Let $a$,$b\in 
\mathbb{R}
^{+},$ $a<b.$ Then, we have%
\begin{equation*}
\left\vert A\left( a^{-1},b^{-1}\right) -L^{-1}\left( a,b\right) \right\vert
\leq \frac{2^{\frac{1}{q}}p}{\left( p+1\right) \left( b-a\right) }\left(
\int_{a}^{b}\left\vert x-\frac{a+b}{2}\right\vert \left\vert x\right\vert
^{-2q}dx\right) ^{\frac{1}{q}}.
\end{equation*}
\end{proposition}

\begin{proof}
The proof is obvious from Theorem \ref{t3} applied for $f(x)=1/x$, $x\in
\lbrack a,b]$.
\end{proof}

\end{document}